\documentclass{article}
\usepackage{amssymb}
\usepackage{amsfonts}
\usepackage{amsmath}

\newtheorem{theorem}{Theorem}

\newtheorem{lemma}[theorem]{Lemma}

\newtheorem{proposition}[theorem]{Proposition}
\newtheorem{remark}[theorem]{Remark}

\newenvironment{proof}[1][Proof]{\noindent\textbf{#1.} }{\ \rule{0.5em}{0.5em}}
\input{tcilatex}

\begin{document}

\title{Differential\ Equations Driven by Gaussian Signals II}
\author{Peter Friz\thanks{%
Corresponding author. Department of Pure Mathematics and Mathematical
Statistics, University of Cambridge. Email: P.K.Friz@statslab.cam.ac.uk. } \
\ \thanks{%
Partially supported by a Leverhulme Research Fellowship. } \and Nicolas
Victoir}
\maketitle

\begin{abstract}
Large classes of multi-dimensional Gaussian processes can be enhanced with
stochastic L\'{e}vy area(s). In a previous paper, we gave sufficient and
essentially necessary conditions, only involving variational properties of
the covariance. Following T. Lyons, the resulting lift to a "Gaussian rough
path" gives a robust theory of (stochastic) differential equations driven by
Gaussian signals with sample path regularity worse than Brownian motion.

The purpose of this sequel paper is to establish convergence of
Karhunen-Loeve approximations in rough path metrics. Particular care is
necessary since martingale arguments are not enough to deal with third
iterated integrals. An abstract support criterion for approximately
continuous Wiener functionals then gives a description of the support of
Gaussian rough paths as the closure of the (canonically lifted)
Cameron-Martin space.
\end{abstract}

\section{Introduction}

Let $X$ be a real-valued centered Gaussian process on $\left[ 0,1\right] $
with continuous sample paths and (continuous) covariance $R=R\left(
s,t\right) =\mathbb{E}\left( X_{s}X_{t}\right) $. We say that $R$ has finite 
$\rho $-variation of $R$, in symbols $R\in C^{\rho \text{$-var$}}\left( %
\left[ 0,1\right] ^{2},\mathbb{R}\right) $ if%
\begin{equation*}
\sup_{D}\sum_{i,j}\left\vert \mathbb{E}\left[ \left(
X_{t_{i+1}}-X_{t_{i}}\right) \left( X_{t_{j+1}}-X_{t_{j}}\right) \right]
\right\vert ^{\rho }<\infty .
\end{equation*}%
An first consequence of this \cite{friz-victoir-07-DEdrivenGaussian_I} is
that (2D) $\rho $-variation of the covariance implies (the usual 1D) $\rho $%
-variation regularity of elements in the associated \textit{Cameron-Martin} 
\textit{space} $\mathcal{H}$, viewed as subspace of $C\left( \left[ 0,1%
\right] ,\mathbb{R}\right) $. Good examples to have in mind are\ standard
Brownian motion with $\rho =1$ and fractional Brownian motion with Hurst
parameter $H\in (0,1/2]$ for which $\rho =1/\left( 2H\right) $. We then
consider a $d$-dimensional, continuous, centered Gaussian process with
independent components,%
\begin{equation*}
X=\left( X^{1},\ldots ,X^{d}\right) ,
\end{equation*}%
with respective covariances $R_{1},\ldots ,R_{d}\in C^{\rho \text{$-var$}}$.
Assuming momentarily smooth sample paths we can define $\mathbf{X}_{\cdot
}\equiv S_{3}\left( X\right) $ by its coordinates in the three
"tensor-levels", $\mathbb{R}^{d},\mathbb{R}^{d}\otimes \mathbb{R}^{d}$and $%
\mathbb{R}^{d}\otimes \mathbb{R}^{d}\otimes \mathbb{R}^{d}$, obtained by
iterated integration%
\begin{equation*}
\mathbf{X}_{\cdot }^{i}=\int_{0}^{\cdot }dX_{r}^{i},\,\,\,\mathbf{X}_{\cdot
}^{i,j}=\int_{0}^{\cdot }\int_{0}^{s}dX_{r}^{i}dX_{s}^{j},\,\,\,\mathbf{X}%
_{\cdot }^{i,j,k}=\int_{0}^{\cdot
}\int_{0}^{t}\int_{0}^{s}dX_{r}^{i}dX_{s}^{j}dX_{t}^{k},\,\,\,
\end{equation*}%
with indices $i,j,k\in \left\{ 1,...,d\right\} $. It turns out that for $%
\rho \in \lbrack 1,2)$ all these integrals make sense as $L^{2}$ limits
(similar to It\^{o}'s theory) and we cite the following

\begin{theorem}[\protect\cite{friz-victoir-07-DEdrivenGaussian_I}]
\label{Thm1Cited}Let $X=\left( X^{1},\ldots ,X^{d}\right) ,Y=\left(
Y^{1},\ldots ,Y^{d}\right) $ be two continuous, centered jointly Gaussian
processes defined on $\left[ 0,1\right] $ such that $\left(
X^{i},Y^{i}\right) $ is independent of $\left( X^{j},Y^{j}\right) $ when $%
i\neq j$. Let $\rho \in \lbrack 1,2)$ and assume the covariance of $\left(
X,Y\right) $ is of finite $\rho $-variation,%
\begin{equation*}
\left\vert R_{\left( X,Y\right) }\right\vert _{\rho -var;\left[ 0,1\right]
^{2}}\leq K<\infty .
\end{equation*}%
Let $p>2\rho $ and $\mathbf{X},\mathbf{Y}$ denote the natural lift\footnote{%
The present result in combination with a Cauchy argument effectively defines
what is meant by natural lift.} of $X,Y$ respectively. Then there exist
positive constants $\theta =\theta \left( p,\rho \right) $ and $C=C\left(
p,\rho ,K\right) $ such that%
\begin{equation*}
\left\vert d_{p\text{$-var$}}\left( \mathbf{X},\mathbf{Y}\right) \right\vert
_{L^{2}}\leq C\left\vert R_{X-Y}\right\vert _{\infty ;\left[ 0,1\right]
^{2}}^{\theta }.
\end{equation*}
\end{theorem}

Let us consider a continuous, centered $d$-dimensional process $W$ with
independent components and finite $\rho \in \lbrack 1,2)$-covariance and its
piecewise linear approximations $Z^{n}$ such that $Z_{t}\equiv Z_{t}^{n}$
for $t\in \left\{ l/n:l=0,...,n\right\} $. The above theorem, applied to $%
\mathbf{X}=S_{3}\left( Z^{n}\right) ,\mathbf{Y}=S_{3}\left( Z^{m}\right) $,
shows that $S_{3}\left( Z^{n}\right) $ is Cauchy and this can be taken as
definition of $\mathbf{W}$, the natural lift of $W$.

\bigskip

Of course, there should be nothing special about piecewise linear
approximations\footnote{%
Observe how the assumptions of theorem \ref{Thm1Cited} rule out McShane's
famous example \cite[Section on Approximations of the Wiener process]%
{ikeda-watanabe-89}.} and we are indeed able to show that Karhunen-Loeve
type approximations of form $S_{3}\left( \mathbb{E}\left( W|\mathcal{F}%
_{n}\right) \right) $ converge to the same natural lift $\mathbf{W}$.
Existence of L\'{e}vy's area for a multi-dimensional Gaussian process and
regularity of Cameron-Martin paths are in fact closely related. Given a
basis of $\mathcal{H}_{i}$, the Cameron-Martin space for $X^{i},$ $i\in
\left\{ 1,\ldots ,d\right\} $, we have the $L^{2}$- or Karhunen-Loeve
expansion%
\begin{equation*}
X^{i}\left( \omega \right) =\sum_{k\in \mathbb{N}}Z_{k}^{i}\left( \omega
\right) h^{i,k}.
\end{equation*}%
If one assumes that $X=\left( X^{1},\ldots ,X^{d}\right) $ lifts to $\mathbf{%
X}$

\begin{itemize}
\item[(i)] having basic symmetry and integrability properties;

\item[(ii)] via a Borel measurable map which coincides with the classical
notion of iterated integrals on smooth paths;

\item[(iii)] such that the lift of $X+h=\left( X^{1}+h^{1},\ldots
,X^{d}+h^{d}\right) $ coincides with the (rough path)\ translation of $%
\mathbf{X}$ by $h$;
\end{itemize}

then a martingale argument \cite{coutin-victoir-2005} shows that L\'{e}vy's
area is given by (the anti-symmetric part of) 
\begin{equation}
\mathbf{X}_{t}^{i,j}=\sum_{k,l}N_{k}^{i}N_{l}^{j}%
\int_{0}^{t}h_{u}^{i,k}dh_{u}^{j,l}  \label{level2KarhunerIntro}
\end{equation}%
and a notion of iterated integral of elements in $\mathcal{H}_{i},\mathcal{H}%
_{j}$ respectively is needed$.$ With a view towards Young integrals, this
underlines the importance of embedding the Cameron-Martin space into a $\rho 
$-variation path space. We will prove that formula (\ref{level2KarhunerIntro}%
) extends to the level of third iterated integrals and insist that
martingale arguments alone are not enough for this and subtle correction
terms need to be taken care of.

\bigskip

As was pointed out in \cite{feyel-pradelle-2006} in the context of
fractional Brownian motion, Karhunen-Loeve convergence + (iii), i.e.
Cameron-Martin perturbations of the lift are given by rough path
translation, imply a support description. In the generality of our
discussion (iii) may fail. Careful revision of the arguments led us to an
abstract support theorem for approximately continuous Wiener functionals;
somewhat similar in spirit to Aida-Kusuoka-Stroock \cite[Cor 1.13]{MR1354161}%
, cf. remark \ref{RemarkAKSetc}.

With a support description of Gaussian rough paths, one has typical rough
path corollaries such as a Stroock-Varadhan type support theorem for
solutions to rough differential equations%
\begin{equation*}
dY=V\left( Y\right) d\mathbf{X}\left( \omega \right) ,\,\,\,V=\left(
V_{1},\dots ,V_{d}\right)
\end{equation*}%
and support description for $S_{N}\left( \mathbf{X}\right) $, the
canonically defined lift of $\mathbf{X}$ to a process with values in the
step-$N$ free nilpotent group with $d$ generators, for any $N\geq 3$. In
contrast to the original motivation of Stroock-Varadhan, there is no
Markovian structure here and hence no applications to maximum principles
etc. Nonetheless, the support description of Gaussian rough path (and
resultingly: of iterated integrals up to any given order) has been a key
ingredient in establishing non-degeneracy of the Malliavin convariance
matrix under H\"{o}rmander's condition \cite{cass-friz-07}.

The construction of a "natural" lift of a class of Gaussian processes
containing fractional Brownian Motion with $H>1/4$ is due to Coutin-Qian, 
\cite{coutin-qian-02} and the condition for $H$ is optimal, \cite%
{lyons-qian-02, coutin-victoir-2005}. Support statements for lifted
fractional Brownian Motion for $H>\frac{1}{3}$ are proved in \cite%
{friz-victoir-04-Note}, \cite{feyel-pradelle-2006}; a Karhunen-Loeve type
approximations for fractional Brownian Motion is studied in \cite%
{millet-2005}. The present paper unifies and generalizes all these results.

\subsection{Notations}

Let $\left( E,d\right) $ be a metric space and $x\in C\left( \left[ 0,1%
\right] ,E\right) $. It then makes sense to speak of $\alpha $-H\"{o}lder-
and $p$-variation "norms" defined as%
\begin{equation*}
\left\Vert x\right\Vert _{\alpha -H\ddot{o}l}=\sup_{0\leq s<t\leq 1}\frac{%
d\left( x_{s},x_{t}\right) }{\left\vert t-s\right\vert ^{\alpha }}%
,\,\left\Vert x\right\Vert _{p-var}=\sup_{D=\left( t_{i}\right) }\left(
\sum_{i}d\left( x_{t_{i}},x_{t_{i+1}}\right) ^{p}\right) ^{1/p}\,\,.
\end{equation*}%
It also makes sense to speak of a $d_{\infty }$-distance of two such paths,%
\begin{equation*}
d_{\infty }\left( x,y\right) =\sup_{0\leq t\leq 1}d\left( x_{t},y_{t}\right)
.
\end{equation*}%
Given a positive integer $N$ the truncated tensor algebra of degree $N$ is
given by the direct sum 
\begin{equation*}
T^{N}\left( \mathbb{R}^{d}\right) =\mathbb{R}\oplus \mathbb{R}^{d}\oplus
...\oplus \left( \mathbb{R}^{d}\right) ^{\otimes N}.
\end{equation*}%
With tensor product $\otimes $, vector addition and usual scalar
multiplication, $T^{N}\left( \mathbb{R}^{d}\right) =\left( T^{N}\left( 
\mathbb{R}^{d}\right) ,\otimes ,+,.\right) $ is an algebra. Functions such
as $\exp $,$\ln :$ $T^{N}\left( \mathbb{R}^{d}\right) \rightarrow
T^{N}\left( \mathbb{R}^{d}\right) $ are defined immediately by their
power-series. Let $\pi _{i}$ denote the canonical projection from $%
T^{N}\left( \mathbb{R}^{d}\right) $ onto $\left( \mathbb{R}^{d}\right)
^{\otimes i}$. Let $p\in \lbrack 1,2)$ and $x\in $ $C^{p\text{-var}}\left( %
\left[ 0,1\right] ,\mathbb{R}^{d}\right) $, the space of continuous $\mathbb{%
R}^{d}$-valued paths of bounded $q$-variation. We define $\mathbf{x}\equiv
S_{N}(x):[0,1]\rightarrow T^{N}\left( \mathbb{R}^{d}\right) $ via iterated
(Young) integration, 
\begin{equation*}
\mathbf{x}_{t}\equiv
S_{N}(x)_{t}=1+\sum_{i=1}^{N}\int_{0<s_{1}<...<s_{i}<t}dx_{s_{1}}\otimes
...\otimes dx_{s_{i}}
\end{equation*}%
noting that $\mathbf{x}_{0}=1+0+...+0=\exp \left( 0\right) \equiv e$, the
neutral element for $\otimes $, and that $\mathbf{x}_{t}$ really takes
values in 
\begin{equation*}
G^{N}\left( \mathbb{R}^{d}\right) =\left\{ g\in T^{N}\left( \mathbb{R}%
^{d}\right) :\exists x\in C^{1\text{-var}}\left( \left[ 0,1\right] ,\mathbb{R%
}^{d}\right) :\text{ }g=S_{N}(x)_{1}\text{ }\right\} ,
\end{equation*}%
a submanifold of $T^{N}\left( \mathbb{R}^{d}\right) $ and, in fact, a Lie
group with product $\otimes $, called the free step-$N$ nilpotent group with 
$d$ generators. Because $\pi _{1}\left[ \mathbf{x}_{t}\right] =x_{t}-x_{0}$
we say that $\mathbf{x}=S_{N}(x)$ is the canonical lift of $x$. There is a
canonical notion of increments,$\,\mathbf{x}_{s,t}:=\mathbf{x}%
_{s}^{-1}\otimes \mathbf{x}_{t}.$The dilation operator $\delta :\mathbb{R}%
\times G^{N}\left( \mathbb{R}^{d}\right) \rightarrow G^{N}\left( \mathbb{R}%
^{d}\right) $ is defined by 
\begin{equation*}
\pi _{i}\left( \delta _{\lambda }(g)\right) =\lambda ^{i}\pi
_{i}(g),\,\,\,i=0,...,N
\end{equation*}%
and a continuous norm on $G^{N}\left( \mathbb{R}^{d}\right) $, homogenous
with respect to $\delta $, the \textit{Carnot-Caratheodory norm}, is given 
\begin{equation*}
\left\Vert g\right\Vert =\inf \left\{ \text{length}(x):x\in C^{1\text{-var}%
}\left( \left[ 0,1\right] ,\mathbb{R}^{d}\right) ,S_{N}(x)_{1}=g\right\} .
\end{equation*}%
By equivalence of continuous, homogenous norms there exists a constant $%
K_{N} $ such that%
\begin{equation*}
\frac{1}{K_{N}}\max_{i=1,...,N}|\pi _{i}(g)|^{1/i}\leq \left\Vert
g\right\Vert \leq K_{N}\max_{i=1,...,N}|\pi _{i}(g)|^{1/i}.
\end{equation*}%
The norm $\left\Vert \cdot \right\Vert $ induces a (left-invariant) metric
on $G^{N}\left( \mathbb{R}^{d}\right) $ known as \textit{Carnot-Caratheodory
metric}, $d(g,h):=\left\Vert g^{-1}\otimes h\right\Vert .$ Now let $x,y\in $ 
$C_{0}\left( [0,1],G^{N}\left( \mathbb{R}^{d}\right) \right) $, the space of
continuous $G^{N}\left( \mathbb{R}^{d}\right) $-valued paths started at the
neutral element $\exp \left( 0\right) =e$. We define $\alpha $-H\"{o}lder-
and $p$-variation distances 
\begin{eqnarray*}
d_{\alpha -H\ddot{o}l}\left( \mathbf{x,y}\right) &=&\sup_{0\leq s<t\leq 1}%
\frac{d\left( \mathbf{x}_{s,t},\mathbf{y}_{s,t}\right) }{\left\vert
t-s\right\vert ^{\alpha }},\,\,\,\,\, \\
d_{p-var}\left( \mathbf{x,y}\right) &=&\sup_{D=\left( t_{i}\right) }\left(
\sum_{i}d\left( \mathbf{x}_{t_{i},t_{i+1}},\mathbf{y}_{t_{i},t_{i+1}}\right)
^{p}\right) ^{1/p}
\end{eqnarray*}%
and also the "$0$-H\"{o}lder" distance, locally $1/N$-H\"{o}lder equivalent
to $d_{\infty }\left( \mathbf{x,y}\right) $,%
\begin{equation*}
d_{0}\left( \mathbf{x,y}\right) =d_{0-H\ddot{o}l}\left( \mathbf{x,y}\right)
=\sup_{0\leq s<t\leq 1}d\left( \mathbf{x}_{s,t},\mathbf{y}_{s,t}\right) .
\end{equation*}%
Note that $d_{\alpha -H\ddot{o}l}\left( \mathbf{x,}0\right) =\left\Vert 
\mathbf{x}\right\Vert _{\alpha -H\ddot{o}l},\,d_{p-var}\left( \mathbf{x,}%
0\right) =\left\Vert \mathbf{x}\right\Vert _{p-var}$ where $0$ denotes the
constant path $\exp \left( 0\right) $, or in fact, any constant path. The
following path spaces will be useful to us

\begin{enumerate}
\item[(i)] $C_{0}^{p-var}\left( \left[ 0,1\right] ,G^{N}\left( \mathbb{R}%
^{d}\right) \right) $: the set of continuous functions $\mathbf{x}$ from $%
\left[ 0,1\right] $ into $G^{N}\left( \mathbb{R}^{d}\right) $ such that $%
\left\Vert \mathbf{x}\right\Vert _{p-var}<\infty $ and $\mathbf{x}_{0}=\exp
\left( 0\right) $.

\item[(ii)] $C_{0}^{0,p-var}\left( \left[ 0,1\right] ,G^{N}\left( \mathbb{R}%
^{d}\right) \right) $: the $d_{p-var}$-closure of 
\begin{equation*}
\left\{ S_{N}\left( x\right) ,x:\left[ 0,1\right] \rightarrow \mathbb{R}^{d}%
\text{ smooth}\right\} .
\end{equation*}

\item[(iii)] $C_{0}^{1/p-H\ddot{o}l}\left( \left[ 0,1\right] ,G^{N}\left( 
\mathbb{R}^{d}\right) \right) $: the set of continuous functions $\mathbf{x}$
from $\left[ 0,1\right] $ into $G^{n}\left( \mathbb{R}^{d}\right) $ such
that $d_{1/p-H\ddot{o}l}\left( 0,\mathbf{x}\right) <\infty $ and $\mathbf{x}%
_{0}=\exp \left( 0\right) .$

\item[(iv)] $C_{0}^{0,1/p-H\ddot{o}l}\left( \left[ 0,1\right] ,G^{N}\left( 
\mathbb{R}^{d}\right) \right) $: the $d_{1/p-H\ddot{o}l}$-closure of 
\begin{equation*}
\left\{ S_{n}\left( x\right) ,x:\left[ 0,1\right] \rightarrow \mathbb{R}^{d}%
\text{ smooth}\right\} .
\end{equation*}
\end{enumerate}

Recall that a geometric $p$-rough path is an element of $C_{0}^{0,p-var}%
\left( \left[ 0,1\right] ,G^{[p]}\left( \mathbb{R}^{d}\right) \right) ,$ and
a weak geometric rough path is an element of $C_{0}^{p-var}\left( \left[ 0,1%
\right] ,G^{[p]}\left( \mathbb{R}^{d}\right) \right) .$ For a detail study
of these spaces and their properties the reader is referred to \cite%
{friz-victoir-04-Note}.

\section{Karhunen-Loeve Approximations}

Any choice of an orthonormal basis in $\mathcal{H}$, say $\left( h^{k}:k\in 
\mathbb{N}\right) $, yields a $L^{2}$-expansion of a Gaussian process $X$ as
(a.s. and $L^{2}$-convergent) sum of the form $X=\sum_{k\in \mathbb{N}%
}Z_{k}h^{k}$ where $Z_{k}:=\tilde{h}^{k}:=\xi \left( h^{k}\right) $ and $%
h\in \mathcal{H}\mapsto \tilde{h}\in L^{2}(\Omega )$ is the classical
isometry between $\mathcal{H}$ and the Gaussian subspace in $L^{2}\left(
\Omega \right) $, sometimes called Paley-Wiener map (see \cite%
{ledoux-talagrand-1991}, \cite{ledoux-1996}, \cite[Chapter 3.4]{DeuSt89} and
the appendix). As a reminder that we work with continuous Gaussian processes
with the concrete index set $\left[ 0,1\right] $, just as for Brownian
motion, we shall refer to $L^{2}$-approximation as Karhunen-Loeve (type)
approximations\footnote{%
Historically, according to a remark in \cite{jain-kallianpur-1970}, the
Karhunen-Loeve approximation corresponds to a specific choice of basis in $%
\mathcal{H}$, obtained from the eigenfunction of a (continuous) covariance
viewed as integral operator.}, in the same spirit as we prefer to call $%
\mathcal{H}$ Cameron-Martin space rather than Reproducing Kernel Hilbert
Space.

As in previous sections, let $X=(X^{i}:i=1,...,d)$ be a centered continuous
Gaussian process, with independent components, each with covariance $R$ of
finite $\rho $-variation for some $\rho \in \lbrack 1,2)$ and dominated by
some 2D control $\omega .$ Let $\mathbf{X}$ be the natural lift of $X$ to a $%
G^{3}\left( \mathbb{R}^{d}\right) $-valued process. If $\mathcal{H}%
_{i}\subset C\left( \left[ 0,1\right] ,\mathbb{R}\right) $ denotes the
Cameron-Martin space associated to $X^{i}$, the Cameron-Martin space to $X$
is identified with $\oplus _{i=1}^{d}\mathcal{H}_{i}$ and if $%
(h_{i}^{k})_{k\geq 1}$ is an orthonormal basis for $\mathcal{H}_{i}$ then $%
\left\{ \left( h_{i}^{k}\left( .\right) \right) _{i=1,..,d},k\geq 1\right\} $
is an orthonormal basis for $\oplus _{i=1}^{d}\mathcal{H}$. \ We can write $%
h^{k}=\left( h_{1}^{k},\cdots ,h_{d}^{k}\right) $.

\subsection{One Dimensional Estimates}

Our object of interest is $X=\left( X^{1},\ldots ,X^{d}\right) $, a $d$%
-dimensional real-valued, centered, continuous Gaussian process with
independent components. Resultingly, the covariance $R=R\left( s,t\right) $
is a diagonal matrix with $d$ entries and in discussing variational
regularity of the covariance of a Karhunen-Loeve approximations we may
assume that $X$ is in fact $1$-dimensional. For any $A\subset \mathbb{N}$ we
define%
\begin{equation*}
\mathcal{F}_{A}=\sigma \left( Z_{k},k\in A\right) ,\,\,\,X_{t}^{A}=\mathbb{E}%
\left[ X_{t}|\mathcal{F}_{A}\right] .
\end{equation*}%
As in \cite{friz-victoir-07-DEdrivenGaussian_I}, $\omega \left( \left[ a,b%
\right] \times \lbrack c,d]\right) $ stands for a (2D) control function
which controls the $\rho $-variation of $R\left( \cdot ,\ast \right) =%
\mathbb{E}\left[ X_{\cdot }X_{\ast }\right] $ over the indicated rectangle,
i.e.%
\begin{equation*}
\left\vert R\right\vert _{\rho \text{-var;}\left[ a,b\right] \times \left[
c,d\right] }^{\rho }\leq \omega \left( \left[ a,b\right] \times \lbrack
c,d]\right) .
\end{equation*}%
It follows from%
\begin{equation*}
\mathbb{E}\left( \left\vert X_{s,t}^{A}\right\vert ^{2}\right) \leq \mathbb{E%
}\left( \left\vert X_{s,t}\right\vert ^{2}\right) \leq \omega \left( \left[
s,t\right] ^{2}\right) ^{1/\rho }
\end{equation*}%
that $X^{A}$ can be taken with continuous sample paths. Moreover, $X^{A}$ is
a Gaussian process in its own right and we shall write $R^{A}$ for its
covariance function,%
\begin{equation*}
R^{A}\left( s,t\right) =\mathbb{E}\left[ X_{s}^{A}X_{t}^{A}\right] \text{.}
\end{equation*}

\begin{lemma}
\label{RegularityOfRK}Assume that $R$ is of finite $\rho $-variation, for
some $\rho \geq 1$. Then if $\min \left\{ \left\vert A\right\vert
,\left\vert A^{c}\right\vert \right\} <\infty $%
\begin{equation*}
\left\vert R^{A}\right\vert _{\rho \text{$-var$}}<\infty .
\end{equation*}%
In particular, if $\rho <2,$ there exists a natural lift of $X^{A}$ to a $%
G^{3}\left( \mathbb{R}^{d}\right) $-valued process denoted by $\mathbf{X}%
^{A} $.
\end{lemma}

\begin{proof}
Assume first $\left\vert A^{c}\right\vert <\infty $. From \cite[Example 7,
Proposition 16]{friz-victoir-07-DEdrivenGaussian_I}, using $\left\vert
h_{k}\right\vert _{\mathcal{H}}=1$, 
\begin{eqnarray*}
\left\vert h_{k}\otimes h_{k}\right\vert _{\rho \text{$-var$;}\left[ s,t%
\right] ^{2}} &\leq &\left\vert \left( h_{k}\right) \right\vert _{p\text{$%
-var$;}\left[ s,t\right] }^{2} \\
&\leq &\left\vert R\right\vert _{2-var;\left[ s,t\right] ^{2}}^{2}.
\end{eqnarray*}%
It follows that%
\begin{eqnarray*}
\left\vert R^{A}\right\vert _{p\text{$-var$;}\left[ s,t\right] ^{2}}
&=&\left\vert R-\sum_{k\in A^{c}}h_{k}\otimes h_{k}\right\vert _{p\text{$%
-var $;}\left[ s,t\right] ^{2}} \\
&\leq &\left\vert R\right\vert _{p\text{$-var$;}\left[ s,t\right]
^{2}}+\sum_{k\in A^{c}}\left\vert h_{k}\otimes h_{k}\right\vert _{p\text{$%
-var$;}\left[ s,t\right] ^{2}} \\
&\leq &\left( 1+m\right) \left\vert R\right\vert _{2-var;\left[ s,t\right]
^{2}}^{2}.
\end{eqnarray*}%
If $A$ is finite, the proof is even easier.
\end{proof}

The interest is in the above lemma is for $\rho \in \lbrack 1,2)$. To obtain
uniform estimates valid for all $A\subset \mathbb{N}$, we have to work in $2$%
-variation (but see remark below).

\begin{lemma}
\label{KL_uniform_2var}Assume that $R$ is of finite $2$-variation. Then, the 
$R^{A}$ has finite $2$-variation, uniformly over all $A\subset \mathbb{N}$.
More precisely,%
\begin{equation*}
\sup_{A\subset \mathbb{N}}\left\vert R^{A}\right\vert _{2\text{$-var$;}\left[
s,t\right] ^{2}}\leq \left\vert R\right\vert _{2\text{$-var$;}\left[ s,t%
\right] ^{2}}.
\end{equation*}
\end{lemma}

\begin{proof}
Let $D=\left( t_{i}\right) $ a subdivision of $\left[ s,t\right] $ and set $%
X_{i}^{A}=X_{t_{i},t_{i+1}}^{A}.$ Let $\beta $ be a positive semi-definite
symmetric matrix, and let us estimate $\left\vert \sum_{i,j}\beta _{i,j}%
\mathbb{E}\left( X_{i}^{A}X_{j}^{A}\right) \right\vert .$ Now 
\begin{equation*}
\mathbb{E}\left( X_{i}^{A}X_{j}^{A}\right) =\sum_{k\in A}\mathbb{E}\left(
Z_{k}X_{i}\right) \mathbb{E}\left( Z_{k}X_{j}\right) =\frac{1}{2}\sum_{k\in
A}\mathbb{E}\left( \left( Z_{k}^{2}-\mathbb{E}\left( Z_{k}^{2}\right)
\right) X_{i}X_{j}\right) ,
\end{equation*}%
so that 
\begin{equation*}
\sum_{i,j}\beta _{i,j}\mathbb{E}\left( X_{i}^{A}X_{j}^{A}\right) =\frac{1}{2}%
\sum_{k\in A}\mathbb{E}\left( \left( Z_{k}^{2}-\mathbb{E}\left(
Z_{k}^{2}\right) \right) \sum_{i,j}\beta _{i,j}X_{i}X_{j}\right) .
\end{equation*}%
As $\beta $ is symmetric, we can write $\beta =P^{T}$\textrm{diag}$\left(
d_{1},\dots ,d_{\#D}\right) P,$ with $PP^{T}$ the identity matrix and
(non-negative) eigenvalues $\left( d_{i}\right) $. By simple linear algebra, 
\begin{equation*}
\sum_{i,j}\beta _{i,j}X_{i}X_{j}=\left( PX\right) ^{T}\mathrm{diag}\left(
\dots \right) \left( PX\right) =\sum_{i}d_{i}\left( PX\right) _{i}^{2}.
\end{equation*}%
and so 
\begin{eqnarray*}
\sum_{i,j}\beta _{i,j}\mathbb{E}\left( X_{i}^{A}X_{j}^{A}\right)
&=&\sum_{k\in A}\sum_{i}d_{i}\frac{1}{2}\mathbb{E}\left( \left( Z_{k}^{2}-%
\mathbb{E}\left( Z_{k}^{2}\right) \right) \left( PX\right) _{i}^{2}\right) \\
&=&\sum_{i}d_{i}\sum_{k\in A}\mathbb{E}\left( Z_{k}\left( PX\right)
_{i}\right) ^{2} \\
&\leq &\sum_{i}d_{i}\mathbb{E}\left( \left( PX\right) _{i}^{2}\right) \text{
\ \ \ \ (Parseval inequality)} \\
&=&\mathbb{E}\left( \left( PX\right) ^{T}D\left( PX\right) \right) \\
&=&\sum_{i,j}\beta _{i,j}\mathbb{E}\left( X_{i}X_{j}\right) \\
&\leq &\left\vert \beta \right\vert _{l^{2}}\left\vert R\right\vert _{2\text{%
$-var$}}.\text{ \ \ \ \ \ \ \ \ (H\"{o}lder inequality)}
\end{eqnarray*}%
Applying this estimate to $\beta _{i,j}=\mathbb{E}\left(
X_{i}^{A}X_{j}^{A}\right) $ we find 
\begin{equation*}
\text{ }\sqrt{\sum_{i,j}\left\vert \mathbb{E}\left(
X_{i}^{A}X_{j}^{A}\right) \right\vert ^{2}}\leq \left\vert R\right\vert _{2%
\text{$-var$}}.
\end{equation*}%
The proof is finished by taking the supremum over all dissections of $\left[
s,t\right] $.
\end{proof}

\begin{remark}
The previous proof can easily extends to showing that if $R$ is of finite $%
\rho $-variation, where $\rho $ is an integer greater than $2$, then for all 
$A\subset \mathbb{N}$ and $s<t$, $\left\vert R^{A}\right\vert _{\rho -var,%
\left[ s,t\right] ^{2}}\leq \left\vert R\right\vert _{\rho -var,\left[ s,t%
\right] ^{2}}.$ This is done by choosing $\beta _{i,j}=\mathbb{E}\left(
X_{i}^{n,m}X_{j}^{n,m}\right) ^{\rho -1}$and indeed if $\rho -1$ $\in 
\mathbb{N}$ then $\beta $ is a positive symmetric matrix (this is a simple
consequence of Hadamard- Schur's lemma). We could not prove (or disprove)
this for general $\rho \geq 1$; with $\beta $ being defined as fractional
Hadamard power,%
\begin{equation*}
\mathbb{E}\left( X_{i}^{n,m}X_{j}^{n,m}\right) ^{\rho -1}\text{sign}\left[ 
\mathbb{E}\left( X_{i}^{n,m}X_{j}^{n,m}\right) \right]
\end{equation*}%
If true for $\rho \in \lbrack 1,2)$, the present convergence results would
have followed directly from Theorem \ref{Thm1Cited}.
\end{remark}

\subsection{Uniform Bounds on the Modulus and Convergence}

We now assume that $R$ has finite $\rho $-variation for some $\rho \in
\lbrack 1,2)$ dominated by some 2D control $\omega $, and we fix $A\subset 
\mathbb{N}$, finite or with finite complement, so that $X^{A}$ admits a
natural $G^{3}\left( \mathbb{R}^{d}\right) $-valued lift, denoted $\mathbf{X}%
^{A}$. Of course, $\mathbf{X}^{\mathbb{N}}=\mathbf{X}$.

\begin{lemma}[Martingale]
\label{martingaleLEvel3}For all $s<t$ in $\left[ 0,1\right] $, the following
equality holds in $g_{3}\left( \mathbb{R}^{d}\right) $,%
\begin{eqnarray*}
\mathbb{E}\left( \ln \left( \mathbf{X}_{s,t}\right) |\mathcal{F}_{A}\right)
&=&\ln \left( \mathbf{X}_{s,t}^{A}\right) \\
&&+\frac{1}{12}\sum_{i\neq j}X_{s,t}^{A;j}R_{X^{A^{c};i}}\left( 
\begin{array}{c}
s \\ 
t%
\end{array}%
,%
\begin{array}{c}
s \\ 
t%
\end{array}%
\right) \left[ e_{i},\left[ e_{i},e_{j}\right] \right] \\
&&-\frac{1}{2}\sum_{i\neq j}\int_{s}^{t}R_{X^{A^{c};i}}\left( 
\begin{array}{c}
u \\ 
t%
\end{array}%
,%
\begin{array}{c}
s \\ 
u%
\end{array}%
\right) dX_{u}^{A,j}\left[ e_{i},\left[ e_{i},e_{j}\right] \right] .
\end{eqnarray*}%
(The integral which appears in the last line is a Young-Wiener integral in
the sense of \cite[Prop. 38]{friz-victoir-07-DEdrivenGaussian_I}).
\end{lemma}

\begin{remark}
Projection to $g_{2}\left( \mathbb{R}^{d}\right) $ yields to pleasant
equality $\mathbb{E}\left( \ln \left( \mathbf{X}_{s,t}\right) |\mathcal{F}%
_{A}\right) =\ln \left( \mathbf{X}_{s,t}^{A}\right) $which explains why
martingale arguments \cite{friz-05}, \cite{friz-victoir-05}, \cite%
{feyel-pradelle-2006}, \cite{coutin-victoir-2005} are enough to discuss the
step-$2$ case. In contrast, the present lemma shows clearly that martingale
arguments are not enough to handle the step-$3$ case.
\end{remark}

\begin{proof}
Our proposition at level $1$ is $\mathbb{E}\left( \pi _{1}\left( \ln \mathbf{%
X}_{s,t}\right) |\mathcal{F}_{A}\right) =\pi _{1}\left( \ln \mathbf{X}%
_{s,t}^{A}\right) ,$ which is (almost) the definition of $X^{A}.$ The
estimate at level $2$ is implies by $\mathbb{E}\left[ \mathbf{X}_{s,t}^{i,j}|%
\mathcal{F}_{A}\right] =\left( \mathbf{X}^{A}\right) _{s,t}^{i,j}.$ This is
fairly straightforward to prove:\ one just need to note that conditioning
equal $L^{2}$-projection is (trivially) $L^{2}$-continuous and recalling
that both $\mathbf{X}$ and $\mathbf{X}^{A}$ are $L^{2}$-limit of lifted
piecewise linear approximations. Level $3$ statements is more complicated.
We can see as above that for distinct indices $i,j,k$%
\begin{equation*}
\mathbb{E}\left[ \mathbf{X}_{s,t}^{i,j,k}|\mathcal{F}_{A}\right] =\left( 
\mathbf{X}^{A}\right) _{s,t}^{i,j,k}.
\end{equation*}%
Hence, from \cite[Prop. 58, Appendix III]{friz-victoir-07-DEdrivenGaussian_I}%
, we see that $\mathbb{E}\left( \ln \left( \mathbf{X}_{s,t}\right) |\mathcal{%
F}_{A}\right) -\ln \left( \mathbf{X}_{s,t}^{A}\right) $ is equal to%
\begin{eqnarray*}
&&\sum_{i\neq j}\mathbb{E}\left( \left. \left\{ \mathbf{X}_{s,t}^{i,i,j}+%
\frac{1}{12}\left\vert X_{s,t}^{i}\right\vert ^{2}X_{s,t}^{j}-\frac{1}{2}%
X_{s,t}^{i}\mathbf{X}_{s,t}^{i,j}\right\} \right\vert \mathcal{F}_{A}\right) %
\left[ e_{i},\left[ e_{i},e_{j}\right] \right] \\
&&-\sum_{i\neq j}\left( \left( \mathbf{X}_{s,t}^{A}\right) ^{i,i,j}+\frac{1}{%
12}\left\vert \left( X^{A}\right) _{s,t}^{i}\right\vert ^{2}\left(
X^{A}\right) _{s,t}^{j}-\frac{1}{2}\left( X^{A}\right) _{s,t}^{i}\left( 
\mathbf{X}^{A}\right) _{s,t}^{i,j}\right) \left[ e_{i},\left[ e_{i},e_{j}%
\right] \right] .
\end{eqnarray*}%
All the three terms can be written as sums (or $L^{2}$-limits thereof)
involving terms of form $X_{r,s}^{i}X_{t,u}^{i}X_{v,w}^{j}$ and since (write 
$X_{r,s}^{i}=X_{r,s}^{A;i}+X_{r,s}^{A^{c};i}$ and similarly for the other
terms)%
\begin{equation*}
\mathbb{E}\left( X_{r,s}^{i}X_{t,u}^{i}X_{v,w}^{j}|\mathcal{F}_{A}\right)
-\left( X^{A}\right) _{r,s}^{i}\left( X^{A}\right) _{t,u}^{i}\left(
X^{A}\right) _{v,w}^{j}=X_{v,w}^{A;j}\mathbb{E}\left(
X_{r,s}^{A^{c};i}X_{t,u}^{^{A^{c}}i}\right) .
\end{equation*}%
\ After integration, we therefore obtain%
\begin{eqnarray*}
\mathbb{E}\left( \mathbf{X}_{s,t}^{i,i,j}|\mathcal{F}_{A}\right) -\left( 
\mathbf{X}_{s,t}^{A}\right) ^{i,i,j} &=&\frac{1}{2}\int_{s}^{t}\mathbb{E}%
\left( \left\vert X_{s,u}^{A^{c};i}\right\vert ^{2}\right) dX_{u}^{A;j} \\
&=&\frac{1}{2}\int_{s}^{t}R_{X^{A^{c};i}}\left( 
\begin{array}{c}
s \\ 
u%
\end{array}%
,%
\begin{array}{c}
s \\ 
u%
\end{array}%
\right) dX_{u}^{A,j},
\end{eqnarray*}%
\begin{eqnarray*}
\mathbb{E}\left( \left\vert X_{s,t}^{i}\right\vert ^{2}X_{s,t}^{j}|\mathcal{F%
}_{A}\right) -\left( X_{s,t}^{A;i}\right) ^{2}X_{s,t}^{A,j} &=&X_{s,t}^{A;j}%
\mathbb{E}\left( \left\vert X_{s,t}^{A^{c},i}\right\vert ^{2}\right) \\
&=&X_{s,t}^{A;j}R_{X^{A^{c};i}}\left( 
\begin{array}{c}
s \\ 
t%
\end{array}%
,%
\begin{array}{c}
s \\ 
t%
\end{array}%
\right)
\end{eqnarray*}%
and%
\begin{eqnarray*}
\mathbb{E}\left( X_{s,t}^{i}\mathbf{X}_{s,t}^{i,j}|\mathcal{F}_{A}\right)
-X_{s,t}^{A;i}\mathbf{X}_{s,t}^{A;i,j} &=&\int_{s}^{t}\mathbb{E}\left(
X_{s,t}^{A^{c};i}X_{s,u}^{A^{c};i}\right) dX_{u}^{A;j} \\
&=&\int_{s}^{t}R_{X^{A^{c};i}}\left( 
\begin{array}{c}
s \\ 
t%
\end{array}%
,%
\begin{array}{c}
s \\ 
u%
\end{array}%
\right) dX_{u}^{A,j}
\end{eqnarray*}%
That concludes the proof.
\end{proof}

\begin{proposition}
For all $s<t,A,i\neq j$, for some constant $C,$%
\begin{equation*}
\left\vert \int_{s}^{t}R_{X^{A^{c};i}}\left( 
\begin{array}{c}
u \\ 
t%
\end{array}%
,%
\begin{array}{c}
s \\ 
u%
\end{array}%
\right) dX_{u}^{A,j}\right\vert _{L^{2}}^{2}\leq C\omega \left( \left[ s,t%
\right] ^{2}\right) ^{3/\rho }\text{.}
\end{equation*}
\end{proposition}

\begin{proof}
From%
\begin{equation*}
\int_{s}^{t}R_{X^{A^{c};i}}\left( 
\begin{array}{c}
u \\ 
t%
\end{array}%
,%
\begin{array}{c}
s \\ 
u%
\end{array}%
\right) dX_{u}^{A,j}=\mathbb{E}\left( \left.
\int_{s}^{t}R_{X^{A^{c};i}}\left( 
\begin{array}{c}
u \\ 
t%
\end{array}%
,%
\begin{array}{c}
s \\ 
u%
\end{array}%
\right) dX_{u}^{j}\right\vert \mathcal{F}_{A}\right)
\end{equation*}%
it suffices to consider the integral with integrator $dX^{j}$. We define%
\begin{equation*}
f\left( u\right) :=R_{X^{A^{c};i}}\left( 
\begin{array}{c}
u \\ 
t%
\end{array}%
,%
\begin{array}{c}
s \\ 
u%
\end{array}%
\right) .
\end{equation*}%
and note that $f\left( s\right) =0$. It is easy to see that for $u<v$ in $%
\left[ s,t\right] $,%
\begin{equation*}
\left\vert f_{u,v}\right\vert ^{2}\leq \left\vert R_{X^{A^{c};i}}\right\vert
_{2\text{$-var$;}\left[ u,v\right] \times \left[ s,t\right] }^{2}+\left\vert
R_{X^{A^{c};i}}\right\vert _{2\text{$-var$;}\left[ s,t\right] \times \left[
u,v\right] }^{2}.
\end{equation*}%
Noting super-additivity of the right hand side in $\left[ u,v\right] $ and
using Lemma \ref{KL_uniform_2var},%
\begin{eqnarray*}
\left\vert f\right\vert _{2\text{$-var$;}\left[ s,t\right] }^{2} &\leq
&2\left\vert R_{X^{A^{c};i}}\right\vert _{2\text{$-var$;}\left[ s,t\right]
^{2}}^{2} \\
&\leq &2\left\vert R_{X^{i}}\right\vert _{2\text{$-var$;}\left[ s,t\right]
^{2}}^{2} \\
&\leq &2\omega \left( \left[ s,t\right] ^{2}\right) ^{2/\rho }.
\end{eqnarray*}%
Now, $f$ has finite $2$-variation and the covariance of the integrator $%
dX^{j}$ has finite $\rho $-variation, $\rho \in \lbrack 1,2)$ controlled by $%
\omega $. Thanks to $1/2+1/\rho >1$ we can conclude with the "Young-Wiener"
estimate \cite[Prop. 38]{friz-victoir-07-DEdrivenGaussian_I}.
\end{proof}

Putting the last two results together and using \cite[Proposition 24]%
{friz-victoir-07-DEdrivenGaussian_I}, we obtain the following theorem.

\begin{theorem}
\label{KHuniformModulus}For all $s<t$ in $\left[ 0,1\right] $ there exists $%
C=C\left( \rho \right) $ such that 
\begin{equation*}
\sup_{A\subset \mathbb{N}\text{,}\min \left\{ \left\vert A\right\vert
,\left\vert A^{C}\right\vert \right\} <\infty }\mathbb{E}\left( \left\Vert 
\mathbf{X}_{s,t}^{A}\right\Vert ^{2}\right) \leq C\omega \left( \left[ s,t%
\right] ^{2}\right) ^{\frac{1}{\rho }}.
\end{equation*}%
For $p>2\rho $ and $\omega \left( \left[ 0,1\right] ^{2}\right) \leq K$
there exists $\eta =\eta \left( p,\rho ,K\right) >0$ such that%
\begin{equation*}
\sup_{A\subset \mathbb{N}\text{,}\min \left\{ \left\vert A\right\vert
,\left\vert A^{C}\right\vert \right\} <\infty }\mathbb{E}\left( \exp \eta
\left\Vert \mathbf{X}^{A}\right\Vert _{p\text{$-var$;}\left[ 0,1\right]
}^{2}\right) <\infty .
\end{equation*}%
If $\omega \left( \left[ s,t\right] ^{2}\right) \leq K\left\vert
t-s\right\vert $ for all $s<t$ in $\left[ 0,1\right] $ we can replace $%
\left\Vert \mathbf{X}^{A}\right\Vert _{p\text{$-var$;}\left[ 0,1\right] }$
by $\left\Vert \mathbf{X}^{A}\right\Vert _{1/p-H\ddot{o}l\text{;}\left[ 0,1%
\right] }$.
\end{theorem}

We now discuss convergence results.

\begin{theorem}
\label{KHmainConvergenceThm}Let $A_{n}=\left\{ 1,...,n\right\} $. For any $%
p>2\rho $ and $q\in \lbrack 1,\infty ),$%
\begin{eqnarray}
d_{p\text{$-var;\left[ 0,1\right] $}}\left( \mathbf{X}^{A_{n}},\mathbf{X}%
\right) &\rightarrow &0\text{ in }L^{q}\left( \Omega \right) \text{ as }%
n\rightarrow \infty ,  \label{KLpvarConvergence} \\
\left\Vert \mathbf{X}^{A_{n}^{c}}\right\Vert _{p\text{$-var;\left[ 0,1\right]
$}} &\rightarrow &0\text{ in }L^{q}\left( \Omega \right) \text{ as }%
n\rightarrow \infty .  \label{KLpvarConvergenceBack}
\end{eqnarray}%
If $\omega $ is H\"{o}lder dominated, i.e. $\sup_{0\leq s<t\leq 1}\omega
\left( \left[ s,t\right] ^{2}\right) /\left\vert t-s\right\vert
^{1/p}<\infty $, then 
\begin{eqnarray}
d_{1/p\text{-$H\ddot{o}l;\left[ 0,1\right] $}}\left( \mathbf{X}^{A_{n}},%
\mathbf{X}\right) &\rightarrow &0\text{ in }L^{q}\left( \Omega \right) \text{
as }n\rightarrow \infty ,  \label{KLHolConvergence} \\
\left\Vert \mathbf{X}^{A_{n}^{c}}\right\Vert _{1/p\text{-$H\ddot{o}l;\left[
0,1\right] $}} &\rightarrow &0\text{ in }L^{q}\left( \Omega \right) \text{
as }n\rightarrow \infty .  \label{KLHolConvergenceBack}
\end{eqnarray}
\end{theorem}

\begin{proof}
\underline{Ad (\ref{KLpvarConvergence}), (\ref{KLHolConvergence}):} From 
\cite[appendix I]{friz-victoir-07-DEdrivenGaussian_I} and Theorem \ref%
{KHuniformModulus}, it is enough to prove that for any fixed $t\in \left[ 0,1%
\right] $, 
\begin{equation*}
d\left( \mathbf{X}_{t}^{A_{n}},\mathbf{X}_{t}\right) \rightarrow 0
\end{equation*}%
in $L^{q}$ or, in fact, in probability (thanks to the uniform $L^{q}$-bounds
for all $q<\infty $ in Theorem \ref{KHuniformModulus}).The topology induced
by $d$ on $G^{3}\left( \mathbb{R}^{d}\right) $ is consistent with the
manifold topology $G^{3}\left( \mathbb{R}^{d}\right) \subset T^{3}\left( 
\mathbb{R}^{d}\right) $ and in particular with the topology induced from the
Euclidean structure on $g^{3}\left( \mathbb{R}^{d}\right) =\ln \left(
G^{3}\left( \mathbb{R}^{d}\right) \right) $, seen as global chart for $%
G^{3}\left( \mathbb{R}^{d}\right) $. It is therefore enough to show for $%
N=1,2,3$ we have pointwise convergence,%
\begin{equation*}
\pi _{N}\left( \ln \left( \mathbf{X}_{t}^{A_{n}}\right) -\ln \left( \mathbf{X%
}_{t}\right) \right) \rightarrow 0\text{ in probability.}
\end{equation*}%
By martingale convergence, this is obvious for $N=1,2$ but for $N=3$ we have
to handle the correction which we identified in Lemma \ref{martingaleLEvel3},%
\begin{equation*}
\left( \frac{1}{12}X_{s,t}^{A;j}R_{X^{A^{c};i}}\left( 
\begin{array}{c}
s \\ 
t%
\end{array}%
,%
\begin{array}{c}
s \\ 
t%
\end{array}%
\right) -\frac{1}{2}\int_{s}^{t}R_{X^{A^{c};i}}\left( 
\begin{array}{c}
u \\ 
t%
\end{array}%
,%
\begin{array}{c}
s \\ 
u%
\end{array}%
\right) dX_{u}^{A,j}\right) \left[ e_{i},\left[ e_{i},e_{j}\right] \right] .
\end{equation*}%
All we need is pointwise convergence in probability to zero of this
expression. Clearly, $X^{A_{n}^{c}}=\mathbb{E}\left[ X|\mathcal{F}_{\left\{
n+1,n+2,...\right\} }\right] \rightarrow 0$ a.s. and in all $L^{q}$ as $%
n\rightarrow \infty $. It follows that $R_{X^{A_{n}^{c};i}}\rightarrow 0$
pointwise which takes care of the first summand. The second term is a
Young-Wiener integral in the sense of \cite[Proposition 38]%
{friz-victoir-07-DEdrivenGaussian_I}. From our uniform estimates and
interpolation, $R_{X^{A_{n}^{c};i}}\rightarrow 0$ in $\left( 2+\varepsilon
\right) $-variation. Using notation from the last proposition,%
\begin{equation*}
\int_{s}^{t}f\left( u\right) dX_{u}^{A_{n},j}=\mathbb{E}\left( \left.
\int_{s}^{t}f\left( u\right) dX_{u}^{j}\right\vert \mathcal{F}%
_{A_{n}}\right) ,
\end{equation*}%
and it is enough to show that $\int_{s}^{t}f\left( u\right)
dX_{u}^{j}\rightarrow 0$ in $L^{2}$. Now,%
\begin{equation*}
\left\vert f\right\vert _{\left( 2+\varepsilon \right) \text{$-var$;}\left[
s,t\right] }^{2+\varepsilon }\leq C\left\vert R_{X^{A_{n}^{c};i}}\right\vert
_{\left( 2+\varepsilon \right) \text{$-var$;}\left[ s,t\right]
^{2}}^{2+\varepsilon }\rightarrow 0
\end{equation*}%
and using the Young-Wiener estimate for $\varepsilon $ chosen small enough
(namely such that $\left( 2+\varepsilon \right) ^{-1}+\rho ^{-1}>1$ which is
always possible since $\rho \in \lbrack 1,2)$) we obtain the required
convergence in $L^{2}$ and hence in probability as required.\newline
\underline{Ad (\ref{KLpvarConvergenceBack}), (\ref{KLHolConvergenceBack}):}
As in the first part of the proof, it is enough to show that, for fixed $%
t\in \left[ 0,1\right] $, $\mathbf{X}_{t}^{A_{n}^{c}}\rightarrow 0$ in
probability or, equivalently,%
\begin{equation*}
\ln \left( \mathbf{X}_{t}^{A_{n}^{c}}\right) \rightarrow 0\text{ in
probability.}
\end{equation*}%
We first claim that $\mathbb{E}\left( \ln \left( \mathbf{X}_{t}\right) |%
\mathcal{F}_{A_{n}^{c}}\right) \rightarrow 0$. Indeed, by backward
martingale convergence and Kolmogorov's 0-1 law,%
\begin{eqnarray*}
\mathbb{E}\left( \ln \left( \mathbf{X}_{t}\right) |\mathcal{F}%
_{A_{n}^{c}}\right) &\rightarrow &\mathbb{E}\left( \ln \left( \mathbf{X}%
_{t}\right) |\cap _{k}\mathcal{F}_{A_{k}^{c}}\right) \text{ a.s. and in all }%
L^{q} \\
\overset{\text{a.s.}}{=}\mathbb{E}\left( \ln \mathbf{X}_{t}\right) &=&0.%
\text{ }
\end{eqnarray*}
Let us detail why the ($g^{3}\left( \mathbb{R}^{d}\right) $-valued)
expectation of $\ln \mathbf{X}_{t}$ is indeed $0$. The only interesting case
is projection to the level $N=3$. By the expansion of $\ln \mathbf{X}_{t}$
given in \cite[Prop. 58]{friz-victoir-07-DEdrivenGaussian_I} we know that $%
\pi _{3}\left( \ln \mathbf{X}_{t}\right) $ involves precisely terms of form $%
\mathbf{X}_{t}^{i,j,k},\mathbf{X}_{t}^{i,i,j},\left( \mathbf{X}%
_{t}^{i}\right) ^{2}\mathbf{X}_{t}^{j}$ and $\mathbf{X}_{t}^{i}\mathbf{X}%
_{t}^{i,j}$ (with disjoint indices $i,j,k\in \left\{ 1,\dots d\right\} $).
Then $\mathbb{E}\mathbf{X}_{t}^{i,j,k}=0$ since $\mathbf{X}_{t}^{i,j,k}$ is
an element of the third homogenous Wiener-It\^{o} chaos. $\mathbb{E}\mathbf{X%
}_{t}^{i,i,j}=0$ since an approximation argument reduces this to the case of
nice sample paths in which case the assertion is clear using Fubini and $%
\mathbb{E}\mathbf{X}^{j}=\mathbb{E}X^{j}=0$). Similar (but easier) for $%
\left( \mathbf{X}_{t}^{i}\right) ^{2}\mathbf{X}_{t}^{j}$. At last $\mathbb{E}%
\left( \mathbf{X}_{t}^{i}\mathbf{X}_{t}^{i,j}\right) =0$ by orthogonality of
the first and second Wiener-It\^{o} chaos.) The proof will be finished if we
can handle the difference between $\ln \left( \mathbf{X}_{t}^{A_{n}^{c}}%
\right) $ and $\mathbb{E}\left( \ln \left( \mathbf{X}_{t}\right) |\mathcal{F}%
_{A_{n}^{c}}\right) $. But using Lemma \ref{martingaleLEvel3}, this is done
in the same way as in the first part of the proof.
\end{proof}

\section{Support Theorem for the Law of $\mathbf{X}$}

We recall the standing assumptions. $X=(X^{i}:i=1,...,d)$ is a centered
continuous Gaussian process on $\left[ 0,1\right] $, with independent
components and finite covariance of finite $\rho \in \lbrack 1,2)$%
-variation, dominated by some 2D control $\omega .$ From \cite[Theorem 35]%
{friz-victoir-07-DEdrivenGaussian_I} we know that, for $p\in \left( 2\rho
,4\right) $, $X$ lifts to a (random) geometric $p$-rough path $\mathbf{X}$
with a.e. sample path in $C_{0}^{0,p\text{$-var$}}\left( \left[ 0,1\right]
,G^{3}\left( \mathbb{R}^{d}\right) \right) $. If $\omega $ is H\"{o}lder
dominated we have sample paths in $C_{0}^{0,1/p-H\ddot{o}l}\left( \left[ 0,1%
\right] ,G^{3}\left( \mathbb{R}^{d}\right) \right) .$

\begin{theorem}
Let $\mathbb{P}_{\ast }\mathbf{X}$ denote the law of $\mathbf{X}$, a Borel
measure on the Polish space $C_{0}^{0,p\text{$-var$}}\left( \left[ 0,1\right]
,G^{3}\left( \mathbb{R}^{d}\right) \right) $. Then%
\begin{equation*}
\text{supp}\left[ \mathbb{P}_{\ast }\mathbf{X}\right] =\text{ }\overline{%
S_{3}\left( \mathcal{H}\right) }
\end{equation*}%
where support and closure are with respect to $p$-variation topology. If $%
\omega $ is H\"{o}lder dominated, $\omega \left( \left[ s,t\right]
^{2}\right) \leq K\left\vert t-s\right\vert $ for some constant $K$, we can
use $1/p$-H\"{o}lder topology instead of $p$-variation topology.
\end{theorem}

\begin{remark}
Note that $S_{3}\left( \mathcal{H}\right) $ is canoncially defined by
iterated Young intgeration. Indeed, any $h\in \mathcal{H}$ has finite $\rho $%
-variation (\cite[Proposition 16]{friz-victoir-07-DEdrivenGaussian_I} and,
under the standing assumption of $\rho \in \lbrack 1,2)$, lifts to a $%
G^{3}\left( \mathbb{R}^{d}\right) $-valued paths (of finite $\rho $%
-variation) by iterated Young integration.
\end{remark}

\begin{proof}
Applying theorem \ref{KHmainConvergenceThm} with $\mathbf{X}^{\left\{
n,n+1,\dots \right\} }$ instead of $\mathbf{X}=\mathbf{X}^{\left\{
1,2,3,\dots \right\} }$ readily shows that for all $n\in \left\{ 1,2,3,\dots
\right\} $,%
\begin{equation*}
d_{p\text{$-var;\left[ 0,1\right] $}}\left( \mathbf{X}^{\left\{ n,n+1,\dots
,m\right\} },\mathbf{X}^{\left\{ n,n+1,\dots ,\right\} }\right) \rightarrow
_{m\rightarrow \infty }0\text{ in probability}
\end{equation*}%
and%
\begin{equation*}
\left\Vert \mathbf{X}^{\left\{ m+1,m+2,\dots \right\} }\right\Vert _{p\text{$%
-var;\left[ 0,1\right] $}}\rightarrow _{m\rightarrow \infty }0.
\end{equation*}%
It is also clear\footnote{%
After all, $\mathbf{X}\left( \omega \right) $ was defined as the the limit
in probability of piecewise linear approximations.} that $\omega \mapsto 
\mathbf{X}\left( \omega \right) $ restricted to $\mathcal{H}$ coincides with 
$S_{3}\left( \cdot \right) $ and is continuous thanks to $\mathcal{H}%
\hookrightarrow C^{\rho \text{-var}}$ and basic continuity properties of
Young integration.We can then conclude with a support description for
abstract Wiener functionals give in the appendix (proposition \ref%
{AbstractSupportCriterion}).
\end{proof}

\begin{remark}
\label{RemarkAKSetc}The idea of proposition \ref{AbstractSupportCriterion}
is to find at least one $\omega \in C\left( \left[ 0,1\right] ,\mathbb{R}%
^{d}\right) $ such that $\mathbf{X}\left( \omega \right) \in $ supp$\left[ 
\mathbb{P}_{\ast }\mathbf{X}\right] $ and such that there exists a sequence $%
\left( g_{n}\right) \subset \mathcal{H}$ such that $\mathbf{X}\left( \omega
-g_{n}\right) $ converges with respect to $d_{p-var}$ to $\mathbf{X}\left(
0\right) \equiv \exp \left( 0\right) \in G^{3}\left( \mathbb{R}^{d}\right) $%
. If elements of $\mathcal{H}$ (or at least some ONB $\left( h_{n}\right)
\subset \mathcal{H}$) have good enough regularity, namely finite $q$%
-variation with $1/p+1/q>1$, then this is relatively easy: One can take%
\begin{equation*}
g_{n}\left( \cdot ,\omega \right) =\sum_{i=1}^{n}\xi \left( h_{k}\right)
|_{\omega }h_{k}\left( \cdot \right) \in C^{q-var};
\end{equation*}%
recalling that $\mathbf{X}\left( \omega \right) $ was constructed as limit
(in probability) of piecewise linear approximations\footnote{%
On the null-set on which piecewise linear approximations do not converge, $%
\mathbf{X}\left( \omega \right) $ is defined as an arbitrary constant.} it
follows by basic properties of the rough path translation operator (using
crucially $1/p+1/q>1$) that for almost every $\omega $%
\begin{eqnarray*}
\mathbf{X}\left( \omega -g_{n}\right) &=&T_{-g_{n}}\mathbf{X}\left( \omega
\right) =\lim_{m\rightarrow \infty }T_{-g_{n}}\mathbf{X}^{\left\{ 1,\dots
,m\right\} }\left( \omega \right) \\
&=&\lim_{m\rightarrow \infty }\mathbf{X}^{\left\{ n+1,\dots ,m\right\}
}\left( \omega \right) \\
&=&\mathbf{X}^{\left\{ n+1,n+2,\dots \right\} }\left( \omega \right)
\end{eqnarray*}%
and by (\ref{KLpvarConvergenceBack}) this converges indeed to $\mathbf{X}%
\left( 0\right) \equiv \exp \left( 0\right) \in G^{3}\left( \mathbb{R}%
^{d}\right) $. This argument (construction of a nice ONB in $\mathcal{H}$)
was used by \cite{feyel-pradelle-2006} in the context of fractional Brownian
motion, using the particular structure of its Volterra kernel. (In fact, as
long as Hurst parameter $H>1/4$, all Cameron-Martin paths have sufficient
variational regularity, cf. \cite{friz-victoir-05-JFA}).\newline
Unfortunately, these arguments breaks down in the general setting when $\rho
\geq 3/2$, since $p=2\rho +\varepsilon $, that is, when dealing with general
Gaussian rough paths of $p$-variation regularity with $p\geq 3$. Our
proposition \ref{AbstractSupportCriterion} is based on a careful revision of
the above; it avoids any use of translation operators and is presented in
its its natural generality: the setting of abstract Wiener functionals.
Support descriptions of abstract Wiener functionals were studied by
Aida-Kusuoka-Stroock and it seems worthwhile to point out that their
abstract support theorem \cite[Cor 1.13]{MR1354161} applies to our situation
when $\rho <3/2$ but not beyond. More precisely, only for $\rho <3/2$ can we
translate $\mathbf{X}$ deterministically in Cameron-Martin directions: then $%
\omega \mapsto \mathbf{X}\left( \omega \right) $ is automatically $\mathcal{K%
}$-continuous in the sense of Aida, Kusuoka and Stroock and the support
description follows from their results. However, their notion of $\mathcal{K}
$-continuity does not cover the regime $\rho \in \lbrack 3/2,2)$; at least
not without addition information (such as the existence of an\ ONB of $%
\mathcal{H}$ contained in $C^{q\text{-var}}$ with $1/p+1/q>1$ ...).
\end{remark}

We recall that such a support description in rough path topology yields, as
consequence of Lyons' limit theorem \cite{lyons-98} and without further
work, a Stroock-Varadhan type support theorem for solutions of rough path
differential equations (RDEs) driven by the (random) geometric $p$-rough
path $\mathbf{X}$. For Brownian motion, this application was first carried
out by Ledoux et al. \cite{LeQiZh02}; they also conjecture the extension to
fractional Brownian motion which was then obtained for $H>1/3$ (i.e. without
third iterated integrals) in \cite{friz-victoir-04-Note} and \cite%
{feyel-pradelle-2006}. It is well-known that the Cameron-Martin space $%
\mathcal{H}^{H}$ of $d$-dimensional fractional Brownian motion started at $0$
contains $C_{0}^{1}\left( \left[ 0,1\right] ,\mathbb{R}^{d}\right) $. From 
\cite[Theorem 35]{friz-victoir-07-DEdrivenGaussian_I}, $\omega $ is H\"{o}%
lder dominated and we may work in $\alpha $-H\"{o}lder topology for $\alpha
\in \lbrack 0,H)$, in particular%
\begin{equation*}
\overline{S_{3}\left( \mathcal{H}^{H}\right) }=C_{0}^{0,\alpha -H\ddot{o}%
l}\left( \left[ 0,1\right] ,G^{3}\left( \mathbb{R}^{d}\right) \right) .
\end{equation*}%
A natural lift of fractional Brownian Motion (also known as \textit{enhanced
fractional Brownian motion}) exists for $H>1/4$ and we see that its support
is full, i.e. equals $C_{0}^{0,\alpha -H\ddot{o}l}\left( \left[ 0,1\right]
,G^{3}\left( \mathbb{R}^{d}\right) \right) $ for any $\alpha \in \lbrack
0,H) $ and $H>1/4$. Remark that the extension to $d$ independent fractional
Brownian motions with different Hurst indices greater than $1/4$ is
immediate.

\section{Appendix}

Let $\mu $ be a mean zero Gaussian measure on a real separable Banach space $%
E$. Following the standard references \cite{ledoux-1996} and \cite[Chapter 4]%
{DeuSt89}, there is an abstract Wiener space factorization of the form%
\begin{equation*}
E^{\ast }\overset{\mathcal{\iota }^{\ast }}{\longrightarrow }L^{2}\left( \mu
\right) \overset{\mathcal{\iota }}{\longrightarrow }E.
\end{equation*}%
Here $\mathcal{\iota }^{\ast }$ denotes the embedding of $E^{\ast }$ into $%
L^{2}\left( \mu \right) $; the notation is justified because $\mathcal{\iota 
}^{\ast }$ can be identified as the adjoint of the map $\mathcal{\iota }$
whose construction we now recall: Let $\left( K_{n}\right) $ be a compact
exhaustion of $E$. If $\varphi \in L^{2}\left( \mu \right) $, then $\mathcal{%
\iota }\left( \varphi I_{K_{n}}\right) $ can be identified with the (strong)
expectation%
\begin{equation*}
\int_{K_{n}}x\varphi \left( x\right) d\mu \left( x\right) \text{.}
\end{equation*}%
This is a\ Cauchy sequence in~$E$ and we denote the limit by $\mathcal{\iota 
}\left( \varphi \right) $. One defines $E_{2}^{\ast }$ to be the closure of $%
E^{\ast }$, or more precisely: $\mathcal{\iota }^{\ast }\left( E^{\ast
}\right) $, in $L^{2}\left( \mu \right) $. The reproducing kernel Hilbert
space\ $\mathcal{H}$ of $\mu $ is then defined as%
\begin{equation*}
\mathcal{H}:=\mathcal{\iota }\left( E_{2}^{\ast }\right) \subset \mathcal{%
\iota }\left( L^{2}\left( \mu \right) \right) \subset E.
\end{equation*}%
The map $\mathcal{\iota }$ restricted to $E_{2}^{\ast }$ is linear and
bijective onto $\mathcal{H}$ and induces a Hilbert structure%
\begin{equation*}
\left\langle h,g\right\rangle _{\mathcal{H}}:=\,\left\langle \tilde{h},%
\tilde{g}\right\rangle _{L^{2}\left( \mu \right) }\text{ \ \ \ }\forall
h,g\in \mathcal{H}
\end{equation*}%
where we set $\tilde{h}\equiv \left( \mathcal{\iota }|_{E_{2}^{\ast
}}\right) ^{-1}\left( h\right) $, of course the meaning of $\tilde{g}$ is
similar. To summarize, we have the picture%
\begin{eqnarray*}
E^{\ast }\overset{\mathcal{\iota }^{\ast }}{\longrightarrow }\mathcal{\iota }%
^{\ast }\left( E^{\ast }\right) &\subset &\overline{\mathcal{\iota }^{\ast
}\left( E^{\ast }\right) }=:E_{2}^{\ast }\subset L^{2}\left( \mu \right) 
\overset{\mathcal{\iota }}{\longrightarrow }E. \\
\text{and }\mathcal{\iota }|_{E_{2}^{\ast }} &:&E_{2}^{\ast
}\longleftrightarrow \mathcal{H}\subset E\text{.}
\end{eqnarray*}%
Under $\mu $, the map $x\mapsto $ $\tilde{h}\left( x\right) $ is a Gaussian
random variable with variance $\left\vert h\right\vert _{\mathcal{H}}^{2}$.
We can think of $x\mapsto X\left( x\right) =x$ as $E$-valued random variable
with law $\mu $. Given an ONB $\left( h_{k}\right) $ in $\mathcal{H}$ we
have the $L^{2}$\textit{-expansion} 
\begin{equation*}
X\left( x\right) =\lim_{m\rightarrow \infty }\sum_{k=1}^{m}\tilde{h}%
_{k}\left( x\right) h_{k}\text{ a.s.}
\end{equation*}%
where the sum converges in $E$ for $\mu $-a.e. $x$ and in all $L^{p}\left(
\mu \right) $-spaces, $p<\infty $. For any $A\subset $\bigskip $\mathbb{N}$
define\footnote{%
If $\left\vert A\right\vert <\infty $, this is a finite sum with values in $%
\mathcal{H}$; if $\left\vert A\right\vert =\infty $ this sum converges in $E$
for $\mu $-almost every $x$ and in every $L^{p}\left( \mu \right) $. All
this follows from $X^{A}$ being the conditional expectation of $X$ given $\{%
\tilde{h}_{k}:k\in A\}.$}%
\begin{equation*}
X^{A}\left( x\right) =\sum_{k\in A}\tilde{h}_{k}\left( x\right) h_{k}.
\end{equation*}

\begin{proposition}
\label{AbstractSupportCriterion}Consider a $L^{2}$\textit{-expansion of }$%
X\left( x\right) $ with respect to an ONB $\left( h_{k}\right) $ in $%
\mathcal{H}$ is such that 
\begin{equation*}
\tilde{h}_{k}\equiv \left( \mathcal{\iota }|_{E_{2}^{\ast }}\right)
^{-1}\left( h_{k}\right) \subset \mathcal{\iota }^{\ast }\left( E^{\ast
}\right)
\end{equation*}%
(rather than $\overline{\mathcal{\iota }^{\ast }\left( E^{\ast }\right) }$,
this is always possible; we think of $\tilde{h}_{k}$ as element in $E^{\ast
} $.) Assume $(M,\rho )$ is a Polish space equipped with its Borel $\sigma $%
-algebra $\mathcal{M}$ and $\mathcal{\varphi }:\left( E,\mathcal{B}\right)
\rightarrow \left( M,\mathcal{M}\right) $ is a measurable map such that%
\footnote{%
Converges in probablility means converges in measure with respect to $\mu $.}
is approximately continuous in the sense that for all $n\in \left\{
1,2,3,\dots \right\} $,%
\begin{equation*}
\rho \left( \mathcal{\varphi }\left( X^{\left\{ n,\dots ,m\right\} }\right) ,%
\mathcal{\varphi }\left( X^{\left\{ n,\dots \right\} }\right) \right)
\rightarrow _{m\rightarrow \infty }0\text{ in probability}
\end{equation*}%
and 
\begin{equation*}
\rho \left( \mathcal{\varphi }\left( X^{\left\{ m+1,\dots \right\} }\right) ,%
\mathcal{\varphi }\left( 0\right) \right) \rightarrow _{m\rightarrow \infty }%
\text{ in probability.}
\end{equation*}%
Assume furthermore that the restriction of $\mathcal{\varphi }|_{\mathcal{H}%
}:\mathcal{H}\rightarrow M$ is continuous. Then the topological support of $%
\mathcal{\varphi }_{\ast }\mu $ is equal to the closure of $\mathcal{\varphi 
}\left( \mathcal{H}\right) $.
\end{proposition}

\begin{proof}
Observing that $X^{\left\{ 1,\dots ,n\right\} }$ is a random but finite
linear combination of elements in $\mathcal{H}$, the inclusion%
\begin{equation*}
\mathrm{supp}\left( \mathcal{\varphi }_{\ast }\mu \right) \subset \mathcal{%
\varphi }\left( \mathcal{H}\right)
\end{equation*}%
is clear. The idea for the converse is to work with 
\begin{equation*}
\hat{\varphi}\left( x\right) :=\lim_{m\rightarrow \infty }\varphi \left(
X^{\left\{ 1,\dots ,m\right\} }\left( x\right) \right) .
\end{equation*}%
Observe that $\hat{\varphi}\equiv \varphi $ on $\mathcal{H}$. Indeed, this
from continuity of $\mathcal{\varphi }|_{\mathcal{H}}:\mathcal{H}\rightarrow
M$ together with%
\begin{equation*}
X^{\left\{ 1,\dots ,m\right\} }\left( h\right) =\sum_{k=1}^{m}\tilde{h}%
_{k}\left( h\right) h_{k}=\sum_{k=1}^{m}\left\langle h_{k},h\right\rangle _{%
\mathcal{H}}h_{k},
\end{equation*}%
which plainly converges in $\mathcal{H}$ to $h$ as $m\rightarrow \infty $.
Moreover, by the first convergence assumption $\hat{\varphi}\left( x\right)
=\varphi \left( X\left( x\right) \right) =\varphi \left( x\right) $ for $\mu 
$-a.e. $x$. We shall prove in the next lemma that there exists at least one
(fixed) element $x\in X$ such that $\mathcal{\hat{\varphi}}\left( x\right)
\in \mathrm{\ }$\textrm{supp}$\left( \mathcal{\varphi }_{\ast }\mu \right) =$
\textrm{supp}$\left( \mathcal{\hat{\varphi}}_{\ast }\mu \right) $ and such
that that there exists $\left( g_{n}\right) \subset H$, which may and will
depend on $x$, such that%
\begin{equation*}
\mathcal{\hat{\varphi}}\left( x-g_{n}\right) \rightarrow \mathcal{\hat{%
\varphi}}\left( 0\right) \text{ as }n\rightarrow \infty .\text{ }
\end{equation*}%
It then follows from the Cameron-Martin theorem, that $\mathcal{\hat{\varphi}%
}\left( x-g_{n}\right) $ and any limit point such as $\mathcal{\hat{\varphi}}%
\left( 0\right) $ is contained in the support. Again, using Cameron-Martin,
any $\mathcal{\hat{\varphi}}\left( h\right) $ with $h\in \mathcal{H}$ will
be in the support and this finishes the converse conclusion.
\end{proof}

\begin{lemma}
Under the assumptions of the previous proposition,%
\begin{equation*}
\mu \left\{ x:\exists \left( g_{n}\right) \subset \mathcal{H}:\hat{\varphi}%
\left( x-g_{n}\right) \rightarrow _{n\rightarrow \infty }\varphi \left(
0\right) \right\} =1.
\end{equation*}%
(Trivially, $\mu \left\{ x:\hat{\varphi}\left( x\right) \in \mathrm{supp}%
\left( \hat{\varphi}_{\ast }\mu \right) \right\} =1$ and the intersection of
these sets has also full measure).
\end{lemma}

\begin{proof}
By assumption, $h_{i}=\mathcal{\iota }\left( \tilde{h}_{i}\right) $ is an
ONB in $\mathcal{H}$ and all $\tilde{h}_{i}$ are of form $\left\langle
\lambda _{i},\cdot \right\rangle _{E^{\ast },E}$. As in \cite[(3.4.14)]%
{DeuSt89},%
\begin{eqnarray*}
\left\langle \lambda _{i},h_{j}\right\rangle _{E^{\ast },E} &=&\left\langle
\left\langle \lambda _{i},\cdot \right\rangle _{E^{\ast },E},\tilde{h}%
_{j}\right\rangle _{L^{2}\left( \mu \right) } \\
&=&\left\langle \tilde{h}_{i},\tilde{h}_{j}\right\rangle _{L^{2}\left( \mu
\right) }=\left\langle h_{i},h_{j}\right\rangle _{\mathcal{H}}.
\end{eqnarray*}%
Thinking of $\tilde{h}_{i}$ as elements on $E^{\ast }$, we abuse of notation
and write $\left\langle \tilde{h}_{i},h_{j}\right\rangle _{E^{\ast },E}$
instead of $\left\langle \lambda _{i},h_{j}\right\rangle _{E,E^{\ast }}$.
What we have just seen is that%
\begin{equation*}
\left\langle \tilde{h}_{i},h_{j}\right\rangle _{E^{\ast },E}=\delta _{i,j}
\end{equation*}%
where $\delta _{i,j}=1$ if $i=j$ and zero otherwise. Let us fix $n$. By
definition, 
\begin{equation*}
\forall x\in E:X^{\left\{ 1,\dots ,m\right\} }\left( x\right)
:=\sum_{i=1}^{m}\left\langle \tilde{h}_{i},x\right\rangle _{E^{\ast },E}h_{i}
\end{equation*}%
and using linearity of the pairing between $E^{\ast }$ and $E$ we find that,
for $m>n$,%
\begin{equation}
\forall x\in E:X^{\left\{ 1,\dots ,m\right\} }\left(
x-\sum_{j=1}^{n}\left\langle \tilde{h}_{j},x\right\rangle _{E^{\ast
},E}h_{j}\right) =\sum_{i=n+1}^{m}\left\langle \tilde{h}_{i},x\right\rangle
_{E^{\ast },E}h_{i}  \label{EqualityForall_x}
\end{equation}%
and so, \textit{again for all} $x\in E$,%
\begin{equation}
\mathcal{\varphi }\left( X^{\left\{ 1,\dots ,m\right\} }\left(
x-\sum_{j=1}^{n}\left\langle \tilde{h}_{j},x\right\rangle _{E^{\ast
},E}h_{j}\right) \right) =\mathcal{\varphi }\left(
\sum_{i=n+1}^{m}\left\langle \tilde{h}_{i},x\right\rangle _{E^{\ast
},E}h_{i}\right) .  \label{FeyelDelAPradProof}
\end{equation}%
By assumption, for every fixed $n\in \mathbb{N}$, the right hand converges
in probability to $\mathcal{\varphi }\left( X^{\left\{ n+1,\dots \right\}
}\left( x\right) \right) $. But then the left hand side must also converges
(in probability) and by a.s. uniqueness of limits in probability,%
\begin{eqnarray*}
\lim_{m\rightarrow \infty }\mathcal{\varphi }\left( X^{\left\{ 1,\dots
,m\right\} }\left( x-\sum_{j=1}^{n}\left\langle \tilde{h}_{j},x\right\rangle
_{E^{\ast },E}h_{j}\right) \right) &=&\lim_{m\rightarrow \infty }\mathcal{%
\varphi }\left( \sum_{i=n+1}^{m}\left\langle \tilde{h}_{i},x\right\rangle
_{E^{\ast },E}h_{i}\right) \\
= &&\mathcal{\varphi }\left( X^{\left\{ n+1,\dots \right\} }\left( x\right)
\right)
\end{eqnarray*}%
for $\mu $-a.e. $x\in E$. The countable union of nullsets still being a
nullset it follows that (i) the set%
\begin{equation*}
\left\{ x:\forall n\in \mathbb{N}\text{:}\lim_{m\rightarrow \infty }\mathcal{%
\varphi }\left( X^{\left\{ 1,\dots ,m\right\} }\left(
x-\sum_{j=1}^{n}\left\langle \tilde{h}_{j},x\right\rangle _{E^{\ast
},E}h_{j}\right) =\mathcal{\varphi }\left( X^{\left\{ n+1,\dots \right\}
}\left( x\right) \right) \right) \right\}
\end{equation*}%
has full $\mu $-measure. By definition of $\mathcal{\hat{\varphi}}$ and our
assumption, $\mathcal{\varphi }\left( X^{\left\{ n+1,\dots \right\} }\left(
x\right) \right) \rightarrow \mathcal{\varphi }\left( 0\right) =\mathcal{%
\hat{\varphi}}\left( 0\right) $ in probability,and hence for a.e. $x$ along
a subsequence $\left( n_{l}\right) $, we see that%
\begin{equation*}
\mu \left\{ x:\mathcal{\hat{\varphi}}\left( x-\sum_{j=1}^{n_{l}}\left\langle 
\tilde{h}_{j},x\right\rangle _{E^{\ast },E}h_{j}\right) \rightarrow
_{l\rightarrow \infty }\mathcal{\hat{\varphi}}\left( 0\right) \right\} =1
\end{equation*}
The proof is finished by defining $\left( g^{l}\right) \subset \mathcal{H}$
by $g^{l}=\sum_{j=1}^{n_{l}}\left\langle \tilde{h}_{j},x\right\rangle
_{E^{\ast },E}h_{j}$.
\end{proof}

\bibliographystyle{plain}
\bibliography{roughpaths}

\def\cprime{$'$} \def\cprime{$'$}
\begin{thebibliography}{10}

\bibitem{MR1354161}
S.~Aida, S.~Kusuoka, and D.~Stroock.
\newblock On the support of {W}iener functionals.
\newblock In {\em Asymptotic problems in probability theory: Wiener functionals
  and asymptotics (Sanda/Kyoto, 1990)}, volume 284 of {\em Pitman Res. Notes
  Math. Ser.}, pages 3--34. Longman Sci. Tech., Harlow, 1993.

\bibitem{cass-friz-07}
Thomas Cass and Peter Friz.
\newblock Densities of rough differential equations under {H}\"ormander's
  condition.
\newblock arXiv:0708.3730, 2007.

\bibitem{coutin-qian-02}
Laure Coutin and Zhongmin Qian.
\newblock Stochastic analysis, rough path analysis and fractional {B}rownian
  motions.
\newblock {\em Probab. Theory Related Fields}, 122(1):108--140, 2002.

\bibitem{coutin-victoir-2005}
Laure Coutin and Nicolas Victoir.
\newblock Enhanced {G}aussian processes and applications.
\newblock preprint, 2005.

\bibitem{DeuSt89}
Jean-Dominique Deuschel and Daniel~W. Stroock.
\newblock {\em Large deviations}, volume 137 of {\em Pure and Applied
  Mathematics}.
\newblock Academic Press Inc., Boston, MA, 1989.

\bibitem{feyel-pradelle-2006}
D.~Feyel and A.~de~La~Pradelle.
\newblock Curvilinear integrals along enriched paths.
\newblock {\em Electronic Journal of Probability}, 11:860--892, 2006.

\bibitem{friz-victoir-05}
Peter Friz and Nicolas Victoir.
\newblock Approximations of the {B}rownian rough path with applications to
  stochastic analysis.
\newblock {\em Ann. Inst. H. Poincar\'e Probab. Statist.}, 41(4):703--724,
  2005.

\bibitem{friz-victoir-04-Note}
Peter Friz and Nicolas Victoir.
\newblock A note on the notion of geometric rough paths.
\newblock {\em Probab. Theory Related Fields}, 136:395--416, 2006.

\bibitem{friz-victoir-05-JFA}
Peter Friz and Nicolas Victoir.
\newblock A variation embedding theorem and applications.
\newblock {\em J. Funct. Anal.}, 239(2):631--637, 2006.

\bibitem{friz-victoir-07-DEdrivenGaussian_I}
Peter Friz and Nicolas Victoir.
\newblock Differential equations driven by {G}aussian signals {I}.
\newblock arXiv:0707.0313, 2007.

\bibitem{friz-05}
Peter~K. Friz.
\newblock Continuity of the {I}t\^o-map for {H}\"older rough paths with
  applications to the support theorem in {H}\"older norm.
\newblock In {\em Probability and partial differential equations in modern
  applied mathematics}, volume 140 of {\em IMA Vol. Math. Appl.}, pages
  117--135. Springer, New York, 2005.

\bibitem{ikeda-watanabe-89}
Nobuyuki Ikeda and Shinzo Watanabe.
\newblock {\em Stochastic differential equations and diffusion processes}.
\newblock North-Holland Publishing Co., Amsterdam, second edition, 1989.

\bibitem{jain-kallianpur-1970}
Naresh~C. Jain and G.~Kallianpur.
\newblock A note on uniform convergence of stochastic processes.
\newblock {\em Ann. Math. Statist.}, 41:1360--1362, 1970.

\bibitem{LeQiZh02}
M.~Ledoux, Z.~Qian, and T.~Zhang.
\newblock Large deviations and support theorem for diffusion processes via
  rough paths.
\newblock {\em Stochastic Process. Appl.}, 102(2):265--283, 2002.

\bibitem{ledoux-1996}
Michel Ledoux.
\newblock Isoperimetry and {G}aussian analysis.
\newblock In {\em Lectures on probability theory and statistics (Saint-Flour,
  1994)}, volume 1648 of {\em Lecture Notes in Math.}, pages 165--294.
  Springer, Berlin, 1996.

\bibitem{ledoux-talagrand-1991}
Michel Ledoux and Michel Talagrand.
\newblock {\em Probability in {B}anach spaces}, volume~23 of {\em Ergebnisse
  der Mathematik und ihrer Grenzgebiete (3) [Results in Mathematics and Related
  Areas (3)]}.
\newblock Springer-Verlag, Berlin, 1991.
\newblock Isoperimetry and processes.

\bibitem{lyons-98}
Terry Lyons.
\newblock Differential equations driven by rough signals.
\newblock {\em Rev. Mat. Iberoamericana}, 14(2):215--310, 1998.

\bibitem{lyons-qian-02}
Terry Lyons and Zhongmin Qian.
\newblock {\em System {C}ontrol and {R}ough {P}aths}.
\newblock Oxford University Press, 2002.
\newblock Oxford Mathematical Monographs.

\bibitem{millet-2005}
Annie Millet and Marta Sanz-Sole.
\newblock Approximation of rough paths of fractional brownian motion, 2007.
\newblock Seminar on Stochastic Analysis, Random Fields and Applications, IV
  ({A}scona, 2005), Progress in Probability, {B}irkhaeuser, {B}asel.

\end{thebibliography}

\bigskip

\end{document}